\documentclass[11pt]{amsart}

\usepackage{euscript}
\usepackage{amsmath}
\usepackage{amsthm}
\usepackage{epsfig}
\usepackage{amssymb}\usepackage{amscd}
\usepackage{epic}
\usepackage{eufrak}

\numberwithin{equation}{section}
\numberwithin{equation}{subsection}

\theoremstyle{plain}
\newtheorem{theorem}[equation]{Theorem}
\newtheorem{lemma}[equation]{Lemma}
\newtheorem{proposition}[equation]{Proposition}

\newtheorem{corollary}[equation]{Corollary}

\newtheorem{conjecture}[equation]{Conjecture}
\newtheorem{thm}[equation]{Theorem}

\newtheorem{prop}[equation]{Proposition}

\theoremstyle{definition}
\newtheorem{example}[equation]{Example}
\newtheorem{remark}[equation]{Remark}

\newtheorem{definition}[equation]{Definition}

\newtheorem{conj}[equation]{Conjecture}

\newtheorem{rem}[equation]{Remark}
\newtheorem{bekezdes}[equation]{}

\numberwithin{equation}{section}
\numberwithin{equation}{subsection}

\def\C{\mathbb C}
\def\Q{\mathbb Q}
\def\R{\mathbb R}
\def\Z{\mathbb Z}

\newcommand{\labelpar}{\label}

\newcommand{\frsw}{\mathfrak{sw}}

\title{Lattice cohomology and rational cuspidal curves}

\author{J\'{o}zsef Bodn\'{a}r}
\address{A. R\'enyi Institute of Mathematics, 1053 Budapest,
Re\'altanoda u. 13-15,  Hungary.}
\email{bodnar.jozef@renyi.mta.hu}

\author{Andr\'as N\'emethi}
\address{A. R\'enyi Institute of Mathematics, 1053 Budapest, Re\'altanoda u. 13-15, Hungary.}
\email{nemethi.andras@renyi.mta.hu}
\thanks{The first author is supported by
 ERC program `LTDBud' at R\'enyi Institute.
 The second  author is partially supported by OTKA Grant 100796. }



\date{}

\begin{document}

\maketitle


\pagestyle{myheadings} \markboth{{\normalsize
J. Bodn\'ar and A. N\'emethi}}{ {\normalsize Lattice cohomology and rational cuspidal curves}}

\begin{abstract}
We show a counterexample to a conjecture of de Bobadilla, Luengo, Melle-Hern\'{a}ndez and N\'{e}methi on rational
 cuspidal projective plane curves, formulated in \cite{BLMN1}. The counterexample is a tricuspidal curve of degree $8$.
On the other hand, we show that if the number of cusps is at most $2$, then the original conjecture can be deduced from
 the recent results of Borodzik and Livingston (\cite{BL1}) and the computations of \cite{NR}.

Then we formulate a `simplified' (slightly weaker)
 version, more in the spirit of the motivation of the original conjecture
(comparing index type numerical invariants),
and we prove it for all currently known rational cuspidal curves.
We make all these identities and inequalities more transparent in the language of
lattice cohomology  $\mathbb{H}^{\ast}(S^3_{-d}(K))$ of surgery 3--manifolds $S^3_{-d}(K)$,
where $K = K_1 \# \cdots \# K_{\nu}$ is a connected sum of algebraic knots.

Finally, we prove that the zeroth lattice cohomology of this surgery manifold, $\mathbb{H}^0(S^3_{-d}(K))$ depends only on the
multiset of multiplicities occurring in the multiplicity sequences describing the
algebraic knots $K_i$. This result is closely
related to the lattice-cohomological reformulation of the above mentioned theorems and conjectures,
and provides new
computational and comparison procedures.
\end{abstract}

\section{Introduction}\labelpar{s:1}

\subsection{}

In \cite{BLMN1} de Bobadilla, Luengo, Melle-Hern\'{a}ndez and N\'{e}methi formulated a conjecture on the topological types of
irreducible singularities of a rational cuspidal projective plane curve $C \subset \mathbb{C}P^2$. Recently in \cite{BL1}
Borodzik and Livingston, mostly motivated by \cite{BLMN1},
 proved a necessary condition satisfied by the topological types of cusps of rational cuspidal plane curves. (They will be
reviewed  in  subsections \ref{ss:1.2} and \ref{ss:H}).
 Both of them cover some deep connection with low--dimensional topology.
 Indeed, the conjecture was motivated by
 another conjecture connecting the Seiberg--Witten invariant of the link of normal surface singularities with the geometric
genus \cite{NN}, while the proof of the main result of \cite{BL1} is based
 on the properties of the $d$--invariant of Heegaard--Floer theory. Our goal is to clarify the
 possible interactions by examples and conceptual theoretical explanations. It turns out that
 this can be done ideally using the theory of lattice cohomology.

In the comparison of the conjecture of \cite{BLMN1} and the theorem of \cite{BL1}, the number of cusps plays
a crucial role.
When there is only one cusp, then  they are equivalent; in particular,  in the unicuspidal case the theorem of
\cite{BL1} proves the conjecture of \cite{BLMN1}. However, in the case of at least two cusps, the connection between the two
conditions is less transparent, much harder.
Although the condition proved in \cite{BL1} contains equalities, while the original conjecture in \cite{BLMN1} contains
inequalities and thus it is seemingly `less precise',
we will see that it is not a combinatorial corollary of the former one
if the number of cusps is at least three.

Nevertheless, after we reformulate all the statements in the language of lattice cohomology (section \ref{s:Lattice}),
we show that for bicuspidal curves the conjecture is
implied by the results of Borodzik and Livingston \cite{BL1} and by the
 lattice cohomology formulae of  \cite{NR}.

Furthermore, we show that
for curves with at least three cusps, the original conjecture is not true, in general.

However, we formulate a weakened version of the conjecture
(more in the spirit of the motivation of the original conjecture, which intended to
connect index type numerical invariants), which quite surprisingly turns out to be true for all known rational cuspidal curves,
even for those with at least three cusps. (This is proved in section \ref{s:Verify}).

In the final section we present a procedure which makes certain lattice
cohomological computations a lot easier: it proves a stability of the lattice cohomology
with respect to some kind of `surgery manipulations' with the multiplicity
sequences of the local singularities. These computations
 are closely related to the lattice cohomological reformulation of the results of \cite{BL1},
 and enlarge drastically and conceptually those geometric situations where the output of
 \cite{BL1} is valid (showing e.g. that the global analytic realization of the local cusp types is `less' important
 among the conditions of the main theorem of \cite{BL1}).
 Accordingly, this also shows that the criterion of \cite{BL1}, as a test for the
 analytic realizability of the rational cuspidal curves, is less restrictive.
More precisely,
the main result of \cite{BL1} is a combinatorial condition on the
 collection of topological types of cusps of existing rational cuspidal projective plane curves.
 This necessary condition can be applied as a criterion when one wants to classify rational cuspidal
 curves. It turns out that this criterion is less restrictive when the number of local topological
 cusp types is larger
 (see Corollary \ref{cor:testweak} and Remark \ref{rem:stability}).

\subsection{Notations and the Conjecture from \cite{BLMN1}.}\label{ss:1.2}
Let $C \subset \mathbb{C}P^2$ be a rational cuspidal curve of degree $d$ with $\nu$ cusps
(that is, with locally irreducible singularities)
at points $P_1, P_2, \dots, P_{\nu}$. The local \emph{embedded topological} type of the singularity
at a point $P_i$ is completely determined by the \emph{semigroup} $\Gamma_i \subset \mathbb{Z}_{\geq 0}$ of the plane curve
singularity $(C,P_i)$, or, equivalently, by the Alexander polynomial $\Delta_i(t)$ of the algebraic knot
$K_i = C \cap S_i \subset S_i$, where $S_i$ is a $3$-sphere centered at $P_i$ with sufficiently small radius.
In our convention $\Delta_i$
is indeed a polynomial, and it is  normalized by $\Delta_i(1)=1$.

A short description of a local topological plane curve singularity type is provided by the \emph{multiplicity sequence} and 
the \emph{Newton pairs}. The \emph{multiplicity sequence} $[n_1, \dots, n_r]$ is a non-increasing sequence of integers, 
obtained by noting the consecutive multiplicities of exceptional divisors occuring in the series of blowups during the 
embedded resolution of the plane curve singularity. 
We will use the short form `$u_n$' for `$u, \dots, u$' ($n$ copies) in the multiplicity sequences
 (e.g. we write $[3_2,2]$ instead of $[3,3,2]$).

The \emph{Newton pairs} $\{(p_k,q_k)\}_{k=1}^r$ (with ${\rm gcd}(p_k, q_k) = 1$, $p_k \geq 2, q_k \geq 1$ and $p_1 > q_1$) 
are useful when one computes the  splice diagram or the Alexander-polynomial of the singularity. 

Both of these invariants also characterize the embedded local topological type completely, see \cite{BK} and \cite{EN}.

By \cite{GDC},
$\Gamma_i$ and  $\Delta_i$ are related as follows:
\begin{equation}\label{eq:delta}
\Delta_i(t) = (1-t)\cdot\sum_{k \in \Gamma_i} t^k.\end{equation}

The \emph{delta-invariant} $\delta_i$  of $(C,P_i)$ is the cardinality
 $ \#\{\mathbb{Z}_{\geq 0}\backslash \Gamma_i \}$.
 Set $\delta := \delta_1 + \dots + \delta_{\nu}$. A necessary condition (coming from the degree-genus formula
for singular curves) for the existence of a degree $d$ rational cuspidal curve with cusps of given topological
type is
 \begin{equation}\label{eq:deltasum}
2\delta = (d-1)(d-2).\end{equation}

Consider the product of Alexander-polynomials: $\Delta(t) := \Delta_1(t)\Delta_2(t) \cdots \Delta_{\nu}(t)$.
There is a unique polynomial $Q$ for which $\Delta(t) = 1 + \delta(t-1) + (t-1)^2Q(t)$.
Write  $Q(t) = \sum_{j=0}^{2\delta-2}q_jt^j$.
For $\nu=1$, using (\ref{eq:delta}) and properties of $\Delta_1$, one shows that (cf. \cite[\S 2]{NR})
\begin{equation}\label{eq:Qsum}
Q(t)=\sum_{s\not\in \Gamma_1} (1+t+\cdots +t^{s-1}),
\ \mbox{hence} \ \ q_j=\#\{s\not\in \Gamma_1 \, :\, s>j\} \ \ \ (\mbox{if} \ \nu=1).
\end{equation}
For arbitrary $\nu$, the dependence of the coefficients of $Q$ in terms of
$\Gamma_i$ will be given in (\ref{eq:QF}).
Notice that
 $q_0=\delta$ and $q_{2\delta-2}=1$ \cite[(2.4.4)]{NR}. From the symmetry of $\Delta$ one also gets
 \begin{equation}\label{eq:symQ}
 q_{2\delta-2-j}=q_j+j+1-\delta \ \ \ \mbox{for \ $0\leq j\leq 2\delta-2$}.
 \end{equation}
Next, set the rational function
\begin{equation}\label{eq:R}
R(t):=\frac{1}{d}\, \sum_{\xi^d=1}\frac{\Delta(\xi t)}{(1-\xi t)^2}-\frac{1-t^{d^2}}{(1-t^d)^3}.\end{equation}
In \cite[(2.4)]{BLMN1} is  proved that  $R(t)$ is a symmetric polynomial ($R(t)=t^{d(d-3)}R(1/t)$), and
\begin{equation}\label{eq:R2} R(t)=\sum_{j=0}^{d-3} \Big( q_{(d-3-j)d}-\frac{(j+1)(j+2)}{2}\Big)\,t^{(d-3-j)d}.\end{equation}

The original conjecture we want to discuss is the following:
\begin{conj} (de Bobadilla, Luengo, Melle-Hern\'{a}ndez, N\'{e}methi, \cite{BLMN1}) \label{conj:blmn}
{\it  For any rational cuspidal plane curve $C \subset \mathbb{C}P^2$ of degree $d$ the coefficients of  $R(t)$ are non--positive.}
\end{conj}
 In the body of the paper (Example \ref{ex:n1}; subsections \ref{ss:3.3} and \ref{ss:3.4}) we show that for $\nu \leq 2$ the computations of \cite{NR}
 reduce the conjecture to the statement of \cite{BL1},  and for $\nu \geq 3$, in general, it is false,
 see Example \ref{ex:counter}.

The main motivation for the expression $R(t)$, and for the formulation of the conjecture was
a weaker version of the statement, a comparison of an analytic invariant (the geometric genus
$p_g$) and a topological invariant (the Seiberg--Witten invariant of the link)
of the superisolated hypersurface singularity associated with $C$.
The authors of \cite{BLMN1} were led to it via the Seiberg--Witten Invariant Conjecture (SWIC) of \cite{NN}.

More precisely, let $f_d$ be the homogeneous equation
of degree $d$ of $C$, and set a generic homogeneous function $f_{d+1}$ of degree $(d+1)$. 
Then $f=f_d+f_{d+1}:(\C^3,0)\to
(\C,0)$ defines an isolated hypersurface singularity (which is superisolated)
with geometric genus
$p_g=d(d-1)(d-2)/6$. Let $L$ denote its oriented link. One shows that it is
the surgery manifold $S_{-d}^3(K)$, where $K$ is the connected sum $\#_iK_i$.
If $\widetilde{X}\to\{f=0\}$ is a resolution, we denote the canonical class of
$\widetilde{X}$ by $K_{\textrm{can}}$ and  ${\rm rank}\,H_2(\widetilde{X})$ by $s_{\widetilde{X}}$.
 Then $K_{\textrm{can}}^2+s_{\widetilde{X}}$ is an invariant of the link (in this case it
equals $-(d-1)(d^2-3d+1)$) and one also has (\cite{BLMN1})
\begin{equation}\label{eq:sw}
R(1)=-\frsw_{\textrm{can}}(L)-(K_{\textrm{can}}^2+s_{\widetilde{X}})/8-p_g.
\end{equation}
Here $\frsw_{\textrm{can}}(L)$ is the Seiberg--Witten invariant of $L$ (associated with the canonical
Spin$^c$--structure). Here we adopt the sign convention of later articles, e.g. of \cite{NB}, which is the opposite of \cite{BLMN1}.
The integer $\frsw_{\textrm{can}}
(L)+(K_{\textrm{can}}^2+s_{\widetilde{X}})/8$ is usually called the
`normalized Seiberg--Witten' invariant. $\frsw_{\textrm{can}}(L)$ can be determined (at least)
by two ways, the first goes via Turaev torsion (as in \cite{BLMN1}), or one can rely on the surgery formula of \cite{NB}.
In both cases the key term is the sum from (\ref{eq:R}).

For `certain singularities' the SWIC predicts that $R(1)=0$. This identity is not true for all
superisolated singularities (cf. \cite{BLMN1} for certain $\nu\geq 2$). Nevertheless, for such germs,
 the  simplified (weaker)
  version of the above conjecture \ref{conj:blmn} can be formulated as follows.
\begin{conj}\label{conj:weak} {\bf (Index theoretical version)}
$$R(1)\leq 0, \ \ \ \mbox{that is,}
 \ \ \ p_g\geq -\frsw(L)-(K_{\textrm{can}}^2+s_{\widetilde{X}})/8.$$ \end{conj}

\subsection{The Hilbert--functions, or the counting functions of the semigroups}\label{ss:H} \

Instead of the semigroup $\Gamma_i$ we will often use its `counting function' $k\mapsto H_i(k)$,
\begin{equation}\label{eq:Hilbert}
H_i(k) := \#\{ s \in \ \Gamma_i : s < k \}.
 \end{equation}
 In fact,  $H_i(k)$ is the coefficient of $t^k$ in the Hilbert--function (with variable $t$) of the local singularity
 $(C,P_i)$.

\begin{example}\label{ex:n1} {\bf (The case $\nu=1$)} In this case
$q_j=\#\{s\not\in\Gamma_1: s>j\}$ (cf. (\ref{eq:Qsum}) or  \cite[\S2]{NR}).
By the symmetry of $\Gamma_1$ (that is,
$s\in\Gamma$ if and only if $2\delta-1-s\not\in\Gamma$) one also has
\begin{equation}\label{eq:QH}
q_{2\delta-2-k}=H_1(k+1) \ \ \ \mbox{for} \ k=0,\ldots, 2\delta-2.
\end{equation}
Hence, the $q$--coefficient needed in (\ref{eq:R2}) is
 $q_{(d-3-j)d}=\#\{s\in\Gamma_1:s\leq jd\}=H_1(jd+1)$.

 Furthermore, the coefficients of $R(t)$
 from equation (\ref{eq:R2}) can be reinterpreted geometrically  by B\'ezout's theorem
 as follows (for details see  \cite[Prop. 2]{BLMN1}).  The dimension of the vector space $V$
 of homogeneous polynomials $G$ of degree $j$  in three variables is $(j+1)(j+2)/2$. Fix $j<d$.
 The number of conditions for $G\in V$ to have with $C$ at $P_1$ intersection multiplicity $> jd$
 is $\#\{s\in \Gamma_1\,:\, s\leq jd\}$. Hence, $H_1(jd+1)<(j+1)(j+2)/2$ would imply the existence
 of a curve with equation $\{G=0\}$ which would contradict B\'ezout's theorem. Therefore,
{\it
if $C \subset \mathbb{C}P^2$ is a rational unicuspidal curve of degree $d$, then the counting
function $H_1$ of the local topological type of its singularity  for each $j = 0, 1, \dots, d-3$
satisfies}
  \[ q_{(d-3-j)d}=H_1(jd+1) \geq  \frac{(j+1)(j+2)}{2}. \]
In particular, this inequality and (\ref{eq:R2}) show that
for $\nu=1$ the  Conjecture \ref{conj:blmn} is equivalent to
 the vanishing of $R(t)$, and also to the  weaker version \ref{conj:weak} (and if $R(1)\leq 0$
 then necessarily $R(1)=0$).
\end{example}

\bekezdes
For arbitrary $\nu$, in terms of our present
 notation, the above inequality transforms into the following general form, cf.
 \cite[Prop. 2]{BLMN1}:
 \begin{lemma}\label{lem:Bezout}
   Let $C \subset \mathbb{C}P^2$ be a rational cuspidal curve of degree $d$ with $\nu$ cusps.
   Then the counting functions $H_i$ $(i = 1, \dots, \nu)$ of the local singularities satisfy
   \begin{equation}\label{eq:INEQB}
   \min\limits_{j_1 + j_2 + \dots + j_{\nu} = jd + 1} \{ H_1(j_1) + H_2(j_2) + \dots + H_{\nu}(j_{\nu}) \} \geq
 \frac{(j+1)(j+2)}{2} \end{equation}
  for each $j = 0, 1, \dots, d-3$.\end{lemma}

 This inequality was improved by Borodzik and Livingston.

\begin{thm} \emph{(Borodzik, Livingston \cite[Theorem 5.4]{BL1})} \label{thmbl}
With the notations of Lemma \ref{lem:Bezout}, in (\ref{eq:INEQB}), in fact, one has equality
 (for each $j = 0, 1, \dots, d-3$):
  \[ \min\limits_{j_1 + j_2 + \dots + j_{\nu} = jd + 1} \{ H_1(j_1) + H_2(j_2) + \dots + H_{\nu}(j_{\nu}) \} =
 \frac{(j+1)(j+2)}{2}. \]
  \end{thm}

It is convenient to reformulate the identity as follows (cf. \cite[5.3]{BL1}). For any two functions $H_1$ and $H_2$
 (defined on integers and bounded from below) we define the \emph{`minimum convolution'}, denoted by $H_1 \diamond H_2$,
 in the following way:
\[ (H_1 \diamond H_2)(j) = \min\limits_{j_1+j_2=j}\{ H_1(j_1) + H_2(j_2) \}.\]
Then from the counting functions $\{H_i\}_{i=1}^\nu$
we construct   $H := H_1 \diamond H_2 \diamond \dots \diamond H_{\nu}$ (it is clear that the operator $\diamond$
is associative and commutative). Then the statement of the previous Theorem \ref{thmbl}
says that  for all $j = 0, 1, \dots, d-3$ one has
\begin{equation} \label{eq:bor_liv}
H(jd+1) = \frac{(j+1)(j+2)}{2}. \end{equation}

It is clear that for $\nu = 1$ (when $H = H_1$) this result implies the original Conjecture \ref{conj:blmn} (hence,
the `index theoretical version' \ref{conj:weak} as well). Keeping in mind Lemma \ref{lem:Bezout}, in fact,
the statements of Conjecture \ref{conj:blmn} and Theorem \ref{thmbl} are  equivalent.

In the last point of \cite[Remark 5.5]{BL1} the authors ask about the relation of the
two statements  for $\nu \geq 2$. We will completely clarify this relation in the next two sections.

We end this section by the following symmetry property of $H$, the analogue of (\ref{eq:symQ}).
\begin{lemma}\label{lem:symH}
$H(2\delta-2-j+1)=H(j+1)-j-1+\delta$ for every $j\in\Z$.
\end{lemma}
\begin{proof}
By the symmetry of each semigroup one gets for each counting function
$H_i(j_i)=H_i(2\delta_i-j_i)+j_i-\delta_i$ for any $j_i\in \Z$. Then use the definition of
$H$.
\end{proof}

\section{Combinatorial comparison of Conjecture \ref{conj:blmn} and Theorem \ref{thmbl}}\label{s:Combinatorial}

\subsection{Reformulation of Conjecture \ref{conj:blmn}.}
Conjecture \ref{conj:blmn}
and the coefficients in equation (\ref{eq:R2})
 resemble the identity  \eqref{eq:bor_liv}. Let us emphasize the difference.

 We start in both cases with the counting functions $H_i$.
 In the Borodzik--Livingston theorem 
 one has to take the  `minimum convolution' $H = H_1 \diamond H_2 \diamond \dots \diamond H_{\nu}$
 and $H(jd+1)$ is compared with $(j+1)(j+2)/2$ in  \eqref{eq:bor_liv}.
 In the conjecture \ref{conj:blmn} first one determines  $\Delta_i$ from
  $H_i$ by (\ref{eq:delta}) and (\ref{eq:Hilbert}). Then one takes the product of all $\Delta_i$,
  and finally one takes the coefficients of $Q$, which is compared with $(j+1)(j+2)/2$.

 Next,  we make explicit these last steps and we provide
  the combinatorial formula for $q_j$.

Define sequences $\{h^{(i)}_j\}_{j=0}^{\infty}$ by $h^{(i)}_j := H_i(j+1)$ (notice the shift by one). For any
sequence $a = \{a_j\}_{j=0}^{\infty}$ denote by $\partial a$ its {\it difference sequence}, i.e.
 $(\partial a)_j = a_j - a_{j-1}$ with the convention that the `$(-1)$st element' of a sequence is always zero,
 i.e. $a_{-1} = 0$. Similarly, we will denote by $\varSigma a$ the {\it sequence of partial sums},
i.e.  $(\varSigma a)_j = a_0 + \dots + a_j$. Of course, $\varSigma \partial a = a$ for any sequence $a$.

By (\ref{eq:delta}) and (\ref{eq:Hilbert}),
the coefficient $c^{(i)}_j$ of $t^j$ in $\Delta_i(t)$ can be
written as $c^{(i)}_j = (\partial \partial h^{(i)})_j$.  

The coefficient sequence of a polynomial product is
the usual convolution of coefficient sequences of the factors. Hence,
 the coefficient $c_j$ of $t^j$ in $\Delta(t)$ is
\[ c_j = \sum_{j_1 + \dots + j_{\nu} = j} c^{(1)}_{j_1} \cdots c^{(\nu)}_{j_{\nu}}. \]
Denoting the convolution of two sequences $a = \{a_j\}_{j=0}^{\infty}$ and $b = \{b_j\}_{j=0}^{\infty}$ by
$a \ast b$, i.e. $(a \ast b)_j = \sum_{k=0}^ja_kb_{j-k}$, we get
$ c_j = (\partial \partial h^{(1)} \ast \dots \ast \partial \partial h^{(\nu)})_j $.
Let us define:
$$F(j) := (\varSigma \varSigma (\partial \partial h^{(1)} \ast \dots \ast \partial
 \partial h^{(\nu)}))_j.$$
If $A(t)=\sum_ja_jt^j$ and $B(t)=\sum_jb_jt^j$ satisfy $A(t)=A(1)+(t-1)B(t)$,
then $(\varSigma a)_j=A(1)-b_j$. This applied twice for $\Delta$ gives
$(\varSigma \varSigma c)_j=j+1-\delta +q_j$. Hence,
 the definition of $Q$ and (\ref{eq:symQ}) provides
\begin{equation}\label{eq:QF}
 q_{2\delta-2-j}  = (\varSigma \varSigma (\partial \partial h^{(1)} \ast \dots \ast \partial \partial
h^{(\nu)}))_j\ =F(j)\ \ \ \mbox{for \ $0 \leq j \leq 2\delta-2$}. \end{equation}
Now we can reformulate the inequalities of Conjecture \ref{conj:blmn}.
\begin{conj}\label{conj:firstref} {\bf (Alternative form of Conjecture \ref{conj:blmn})} \

{\it
Let $C \subset \mathbb{C}P^2$ be a rational cuspidal curve of degree $d$ with $\nu$ cusps of given
topological types (in particular, $d(d-3) = 2\delta-2$).
Set $F(j) := (\varSigma \varSigma (\partial \partial h^{(1)} \ast \dots \ast \partial \partial h^{(\nu)}))_j$, where
 $h^{(i)}_j=H_i(j+1)$, and $H_i$ is the semigroup counting function of the $i$--th singularity. Then}
\begin{equation}\label{eq:u}
F(jd) \leq \frac{(j + 1)(j + 2)}{2} \ \ \
\mbox{for all $j = 0, 1, \dots, d-3$}.\end{equation}\end{conj}
The `index  version' is obtained by taking sum.
\begin{conj}\label{conj:weakalt} {\bf (Index theoretical version, first alternative form)}\ 

{\it Under the conditions and with the notation of \ref{conj:firstref},}
\[ \sum_{j=0}^{d-3} F(jd) \leq  \sum_{j=0}^{d-3} \frac{(j+1)(j+2)}{2}=
\frac{d(d-1)(d-2)}{6}. \]
\end{conj}

\subsection{Examples and counterexamples.}

Let us summarize the situation. Starting from the semigroups of $\nu$ local singularities we can define integral
functions $H$ and $F$ depending only on the local topological types of the singularities.

\begin{definition}\label{def:candidates}
If the sum  of delta invariants of the local singularity types, $\delta$,
is of form $2\delta =(d-1)(d-2)$
 for some integer $d$, we say that these $\nu$ local topological types are \emph{candidates} to be the
$\nu$ singularities of a rational cuspidal plane curve of degree $d$.
\end{definition}

If such a curve exists  then we
 have a theorem stating `each $d$--th value' of  $H$,
 and also a conjecture giving an upper bound on `each $d$--th value' of  $F$.

As we already mentioned, if $\nu = 1$ then  $F(k) = H(k+1)$ for each $k\in\Z_{\geq 0}$
 (not just for $k\in d \cdot \Z_{\geq 0}$), and the theorem implies the conjecture (and
 the inequalities are equalities).

However, for $\nu > 1$ the values $F(k)$ and $H(k+1)$ become different. If one starts to play with
{\it two} singularities, one can notice that $F(k) \leq H(k+1)$ seem to be true  (for every 
integer $k\geq 0$, not just for the multiples of $d$). [Later, using lattice cohomology interpretations, we will
 prove that this is indeed true for $\nu = 2$, cf. \ref{ss:3.4}.]
With these  facts in mind, it is tempting to conjecture that maybe {\it the inequality
 $F(k) \leq H(k+1)$ is always true --- even
independently of $d$ ---}
which would be an interesting, completely combinatorial statement making Conjecture \ref{conj:blmn} being
 a far more weaker corollary of Theorem \ref{thmbl}. But, for $\nu \geq 3$ there is no such relation between
functions $F$ and $H$, as we will demonstrate next.

\begin{example} \label{ex:first222}
Take $\nu=3$, and assume that all local singularities are `simple' cusps, that is cusps
with multiplicity seqence $[2]$ (or, equivalently, with one Newton pair $(2,3)$, or with semigroup
$\langle2,3\rangle = \{0, 2, 3, 4, \dots\}$). Then the functions $F$ and $H$ are as follows:

\vspace{2mm}

\begin{center}
\begin{tabular}{|c | c | c | c | c | c |} \hline
  $k$ & $0$ & $1$ & $2$ & $3$ & $4$ \\ \hline
  $H(k+1)$ & $1$ & $1$ & $2$ & $2$ & $3$ \\ \hline
  $F(k)$ & $1$ & $-1$ & $3$ & $0$ & $3$ \\ \hline
  $H(k+1)-F(k)$ & $0$ & $2$ & $-1$ & $2$ & $0$ \\ \hline
\end{tabular}
\end{center}

\vspace{2mm}

\noindent Notice that for $k = 2$ the desired inequality $F(k) \leq H(k+1)$  fails.
Hence the inequality $F(k) \leq H(k+1)$ cannot be true for {\it any integer } $k$.
(By the way, this collection of local cusp types can be realized on a rational tricuspidal curve of degree four, cf. Proposition \ref{prop:tricuspidal}.)

\end{example}

\begin{example}
Consider now a tricuspidal rational projective plane curve of degree $d = 5$ such that each of its singularities has one
Newton pair, namely  $(3,4)$, $(2,5)$ and $(2,3)$, or multiplicity sequences $[3], [2_2]$ and $[2]$, respectively (cf. Proposition \ref{prop:tricuspidal}). So each semigroup is generated (over $\mathbb{Z}_{\geq 0}$) by
 the corresponding two integers.
A computation shows that the function values are as follows:

\vspace{2mm}

\begin{center}
\begin{tabular}{|c | c | c | c | c | c | c | c | c | c | c | c |} \hline
  $k$ & $0$ & $1$ & $2$ & $3$ & $4$ & $5$ & $6$ & $7$ & $8$ & $9$ & $10$ \\ \hline
  $H(k+1)$ & $1$ & $1$ & $1$ & $2$ & $2$ & $3$ & $3$ & $4$ & $4$ & $5$ & $6$ \\ \hline
  $F(k)$ & $1$ & $-1$ & $2$ & $0$ & $2$ & $1$ & $3$ & $2$ & $5$ & $3$ & $6$ \\ \hline
  $H(k+1)-F(k)$ & $0$ & $2$ & $-1$ & $2$ & $0$ & $2$ & $0$ & $2$ & $-1$ & $2$ & $0$ \\ \hline
\end{tabular}
\end{center}

\vspace{2mm}

\noindent
Since the data are provided by
an existing curve of degree $5$, the $H(jd + 1)$-values at $jd=k = 0, 5, 10$ are the corresponding triangular numbers
($1, 3, 6$, respectively), as predicted by Theorem \ref{thmbl}. Notice also that again, the desired inequality
$F(k) \leq H(k+1)$ fails at $k = 2, 8$. However, the inequalities needed
for Conjecture \ref{conj:blmn} corresponding to  $k = 0, 5, 10$ (multiples of $d$) are
 true. (This example appears in \cite{BLMN1}, supporting Conjecture \ref{conj:blmn}.)
 So one could still hope in the inequality $F(k) \leq H(k+1)$ in the case of existing curves and for $k \in d \cdot \mathbb{Z}$, where $d$ is the degree of the curve.
\end{example}

\begin{example}\label{ex:counter}{\bf (Counterexample to Conjecture \ref{conj:blmn})}

 Consider three semigroups given by two generators each as follows: $\Gamma_1 = \left\langle 6, 7 \right\rangle$,
 $\Gamma_2 = \left\langle 2, 9 \right\rangle$ and $\Gamma_3 = \left\langle 2, 5 \right\rangle$. These are semigroups of
 plane curve singularities characterized by multiplicity sequences $[6]$, $[2_4]$ and $[2_2]$, respectively.
 There exists a rational tricuspidal curve of degree $d = 8$ with three singularities exactly of this topological type
(cf. Proposition \ref{prop:tricuspidal}).
 The values of functions $H$ and $F$ are as follows (we are interested only in the values at the multiples of $d$):

\vspace{2mm}

\begin{center}
\begin{tabular}{|c | c | c | c | c | c | c | c | c | c | c | c |} \hline
  $k$ & $0$ & $\dots$ & $8$ & $\dots$ & $16$ & $\dots$ & $24$ & $\dots$ & $32$ & $\dots$ & $40$ \\ \hline
  $H(k+1)$ & $1$ & $\dots$ & $3$ & $\dots$ & $6$ & $\dots$ & $10$ & $\dots$ & $15$ & $\dots$ & $21$ \\ \hline
  $F(k)$ & $1$ & $\dots$ & $4$ & $\dots$ & $5$ & $\dots$ & $9$ & $\dots$ & $16$ & $\dots$ & $21$ \\ \hline
  $H(k+1)-F(k)$ & $0$ & $\dots$ & $-1$ & $\dots$ & $1$ & $\dots$ & $1$ & $\dots$ & $-1$ & $\dots$ & $0$ \\ \hline
\end{tabular}
\end{center}

\vspace{2mm}

Of course Theorem \ref{thmbl} is satisfied (we see the triangular numbers in the second row). The condition
$\sum_{j=0}^{d-3} F(jd) \leq \sum_{j=0}^{d-3} \frac{(j+1)(j+2)}{2}$, i.e. the  (index theoretical version)
  Conjecture \ref{conj:weakalt}
 is also satisfied (by summation of the  last row, in fact, by equality).

 However, for $j = 1$ and $j = 4$ the inequality $F(jd) \leq \frac{(j+1)(j+2)}{2}$ fails,
 hence {\it this is a counterexample to Conjecture \ref{conj:blmn}.}

This example was not checked in \cite{BLMN1},
as it was not clear at that time that the number of cusps
 was crucial.  Whole series with $\nu = 1$  were checked and
 other examples (also with $\nu \geq 3$), but only up to
degree $7$ (note that a complete classification of cuspidal curves exists only up to degree $6$).
As we will see later in Remark \ref{rem:detailed} the smallest degree where Conjecture \ref{conj:blmn} fails among currently
 known rational cuspidal curves is exactly degree $8$.

This example also shows that the inequalities of the Conjecture \ref{conj:blmn}
are {\it not combinatorial consequences} of the equalities of Theorem \ref{thmbl}.
Moreover, the inequalities $F(k) \leq H(k+1)$ are not true in general, not even for existing
curves of degree $d$ and setting $k \in d \cdot \mathbb{Z}$.
\end{example}

\begin{example}
The following example will show that 
the weakened version, Conjecture \ref{conj:weakalt} is
{\it not a combinatorial consequence} of the equalities of Theorem \ref{thmbl} either.

Consider three semigroups given by their generators:
 $\Gamma_1 = \left\langle 3, 5\right\rangle$, $\Gamma_2 = \left\langle 2, 3\right\rangle$, $\Gamma_3 =
 \left\langle 2, 3\right\rangle$. (The corresponding multiplicity sequences are $[3,2], [2], [2]$, respectively.) 
The sum of delta-invariants is $4 + 1 + 1 =(5-1)(5-2)/2$, so these three topological
types of local singularities are possible candidates for three cusps of a tricuspidal rational projective plane curve of
 degree $d = 5$. Since the complex projective quintics are completely classified, it is known that such a curve does not exist.
However, Theorem \ref{thmbl} does not exclude the existence of this curve, as the values $H(k+1)$ are again
$1, 3, 6$ at $k = 0, 5, 10$, respectively.

\vspace{2mm}

\begin{center}
\begin{tabular}{|c | c | c | c | c | c | c | c | c | c | c | c |} \hline
  $k$ & $0$ & $1$ & $2$ & $3$ & $4$ & $5$ & $6$ & $7$ & $8$ & $9$ & $10$ \\ \hline
  $H(k+1)$ & $1$ & $1$ & $1$ & $2$ & $2$ & $3$ & $3$ & $4$ & $4$ & $5$ & $6$ \\ \hline
  $F(k)$ & $1$ & $-1$ & $2$ & $1$ & $0$ & $4$ & $1$ & $3$ & $5$ & $3$ & $6$ \\ \hline
  $H(k+1)-F(k)$ & $0$ & $2$ & $-1$ & $1$ & $2$ & $-1$ & $2$ & $1$ & $-1$ & $2$ & $0$ \\ \hline
\end{tabular}
\end{center}

\vspace{2mm}

On the other hand, the inequality of Conjecture \ref{conj:blmn} fails at $k = 5$. (Of course it does not mean anything
 from the point of view of the question of existence of this curve, as in the previous example it also failed for an
existing curve.) But, additionally,
for this candidate the weakened version \ref{conj:weakalt} also fails.

Therefore, if Conjecture \ref{conj:weakalt} would be proved (independently of the classification of projective curves), 
it would provide {\it an independent tool for checking
 wether a given collection of local topological singularity types can be realized as the collection of cusp types of a
rational cuspidal projective plane curve.}
\end{example}

\section{Lattice cohomological comparison}\label{s:Lattice}

\subsection{}
Now we show the lattice-cohomological meaning of the values of functions $H$ and $F$. From this point of view,
it will be obvious that, on one hand, for $\nu = 2$ inequalities $F(k) \leq H(k + 1)$ hold, but on the other hand,
we cannot expect such a relation for $\nu \geq 3$. The necessary computations are done in \cite{NR}, but they were not
analyzed  from the present point of view.

\bekezdes \label{latticeshort} For the definition of lattice cohomology, see \cite{Nlattice}. There is a detailed description in Section 3 of \cite{NR} as well. In short, the construction is the following.

Usually one starts with a lattice $\Z^s$ with fixed base elements $\{E_i\}_i$. This automatically
provides a cubical decomposition of $\R^s=\Z^s\otimes \R$: the 0--cubes are the lattice points
$l\in \Z^s$, the 1--cubes are the `segments' with endpoints $l$ and $l+E_i$, and more generally,
a $q$--cube $\square=(l,I)$ is determined by a lattice point $l\in \Z^s$ and a subset $I\subset \{1,\ldots,s\}$
 with $\# I=q$, and it has vertices at the lattice points $l+\sum_{j\in J}E_j$ for different $J\subset I$.

 One also takes a weight function $w:\Z^s\to \Z$ bounded below,
  and for each cube $\square=(l,I)$ one takes
 $w(\square):=\max\{w(v), \ \mbox{$v$ vertex of $\square$}\}$.  Then, for each integer $n\geq \min(w)$
 one considers the simplicial complex $S_n$, the union of all the cubes (of any dimension)
 with $w(\square)\leq n$. Then the {\it lattice cohomology associated with $w$} is
 $\{\mathbb{H}^q(\Z^s,w)\}_{q\geq 0}$, defined by $\mathbb{H}^q(\Z^s,w):=\oplus _{n\geq \min(w)}
 H^q(S_n,\Z)$. Each $\mathbb{H}^q$ is graded (by $n$) and it is a $\Z[U]$--module, where
 the $U$--action consists of  the restriction maps induced by the inclusions $S_n\hookrightarrow S_{n+1}$.
 Similarly, one defines the {\it reduced cohomology associated with $w$} by
  $\mathbb{H}_{{\rm red}}^q(\Z^s,w):=\oplus _{n\geq \min(w)}
 \tilde{H}^q(S_n,\Z)$. In all our cases $\mathbb{H}_{{\rm red}}^q(\Z^s,w)$ has finite $\Z$--rank.
The normalized Euler characteristic of  $\mathbb{H}^*(\Z^s,w)$ is ${\rm eu}\,\mathbb{H}^*:=
-\min(w)+\sum_{q\geq 0}\, (-1)^q\,{\rm rank}_\Z\, \mathbb{H}^q_{{\rm red}}$. Formally, we also set
${\rm eu}\,\mathbb{H}^0:=-\min(w)+{\rm rank}_\Z\, \mathbb{H}^0_{{\rm red}}$.

\bekezdes
In \cite{NR} the authors compute lattice cohomologies of certain $3$-manifolds. The input consists of $\nu$ local
topological plane curve singularity types, as in our case, and a positive integer $d$. The $3$--manifold studied
in \cite{NR} is $S^3_{-d}(K)$,  the manifold obtained by a $(-d)$--surgery along
the connected sum $K$ of knots of the given plane curve singularities ($K = K_1 \# \cdots \# K_{\nu}\subset S^3$).
For motivation see subsection \ref{ss:1.2}. (However,  we do not assume here
that $(d-1)(d-2)=2\delta$.)

$S^3_{-d}(K)$ is a plumbed 3--manifold represented by  a negative definite plumbing graph. For such
plumbing manifolds, the lattice considered in the above construction is freely generated by the
vertices of the graph. On the other hand, for each ${\rm Spin}^c$--structure one defines a
weight function. In the present case, $H_1(S^3_{-d}(K))=\Z_d$, hence one has $d$ different
 ${\rm Spin}^c$--structures, parametrized by $a\in\{0, \ldots, d-1\}$. For each $a$
 one defines a weight function $w_a$, hence a lattice cohomology $\mathbb{H}^*(w_a)$.
 It turns out that the cohomology is independent of the plumbing representation; it depends only on
 $S^3_{-d}(K)$ and $a$.  We denote it by $\mathbb{H}^*(S^3_{-d}(K),a)$.

However, in practice, the lattice cohomologies are not computed by the definition presented above,
but by a powerful general machinery,  called `lattice reduction' (see
\cite{NR} or \cite{LN}). This allows to express the cohomology modules in a lattice (in fact, in a
`rectangle') of rank $\nu$
directly from the semigroups (or counting functions) of the given local topological singularity types.


The main result of this section is the following theorem.

\begin{theorem}\label{rem:generalformula}
Assume that $K=\#_iK_i$, where $\{K_i\}_i$ are algebraic knots.
Define functions $H$ and $F$ as in previous sections. Assume that $d$ is any positive integer. Then
\begin{equation}\label{eq:generalh0}
{\rm eu\, } \mathbb{H}^0\left(S^3_{-d}(K),a\right) = \sum_{\substack{j\equiv a ({\rm mod\ } d)
\\ 0\leq j \leq 2\delta-2}} \left( H(j+1) + \delta-1-j \right),
\end{equation}
\begin{equation}\label{eq:generalhq}
{\rm eu\, } \mathbb{H}^{\ast}\left(S^3_{-d}(K),a\right) = \sum_{\substack{j \equiv a ({\rm mod\ } d)
\\ 0\leq j \leq 2\delta-2}} \left( F(j) + \delta-1-j \right).
\end{equation}
\end{theorem}
\begin{proof}
We will recall several needed statements from \cite{NR}.

Let $f_i$ be the local equation of $(C,P_i)$, and let
 $m_i$ be the multiplicity along the unique $(-1)$--irreducible exceptional
 divisor of the pull back of $f_i$
in  the minimal good embedded resolution of $(C,P_i)$. It is a topological  invariant, and $m_i >2 \delta_i$.

We consider the lattice points  in the rank--$\nu$ multirectangle
$R:= [0,m_1] \times \dots \times [0, m_{\nu}]$.
We denote them  by  ${\bf x}=(x_1,\ldots, x_\nu)$, and we also write $|{\bf x}|:=\sum_{i=1}^\nu x_i$.

For any $a$ (with $0\leq a\leq d-1$) we set the
weight function on $R$ by
\[ w_a(\mathbf{x})=  \sum_{i=1}^{\nu} H_i(x_i) + \textrm{min}\{0, 1 +a - |{\bf x}|\}. \]
It is convenient to define another weight function too, which is independent of $d$ and $a$:
\[ W(\mathbf{x}) = \sum_{i=1}^{\nu}\#\{ s \notin \Gamma_i : s \geq x_i \}
= \sum_{i=1}^{\nu}\left( \delta_i - x_i + H_i(x_i) \right)= \delta-|{\bf x}|+
\sum_{i=1}^{\nu}H_i(x_i). \]
For any $j\geq 0$ denote the `diagonal hyperplanes' of the multirectangle by
\[ T_j := \{ {\bf x}\in R \, :\,  |{\bf x}| = j + 1 \}. \]
Note that $T_j=\emptyset$ whenever   $j>M := m_1 + \dots + m_{\nu}$.

To formulate the next result, we
define lattice cohomologies on the `diagonal' sets $T_j$ as well, considering the cohomologies of the intersection of simplicial level sets of the lattice rectangle and the $(\nu-1)$-dimenisonal hyperplane of $T_j$, i.e. $\mathbb{H}_{red}^q(T_j,W):=\oplus _{n\geq \min(W)} \tilde{H}^q(S_n \cap T_j,\Z)$ (and similarly for the non-reduced version; cf. \cite[(6.1.10)]{NR}), where the simplicial complex (level set) $S_n$ is the union of all cubes $\square$ with $W(\square) \leq n$.

\begin{thm} \emph{(\cite{NR}, formulae (6.1.15) and  (6.1.16))} \label{thm_sis_eu}
For any $d>0$ one has:
\begin{equation}\label{eq:NR-1} {\rm eu\ } \mathbb{H}^0\left(S^3_{-d}(K),a\right) = \sum_{\substack{j \equiv a ({\rm mod\ } d),\\
 0\leq j\leq M }} \min W\lvert_{T_j}, \end{equation}


 \begin{equation}\label{eq:NR0}
{\rm eu\ } \mathbb{H}^{\ast}\left(S^3_{-d}(K),a\right) = - \sum_{\substack{j \equiv a ({\rm mod\ } d), \\
0\leq j\leq M}} {\rm eu\ } \mathbb{H}^{\ast}\left(T_j, W\right). \end{equation}
\end{thm}
Clearly $\min W|_{T_j}=\delta-j-1+H(j+1)$, which equals
$H(2\delta -1-j)$ by Lemma \ref{lem:symH}, thus it is zero for $j\not\leq 2\delta-2$.
Hence the identity (\ref{eq:generalh0}) follows.

Next, fix some $j\geq 0$, and apply Theorem \ref{thm_sis_eu} for an auxiliary large
 $D>M$ (substituted for $d$),  and for $a=j$.
By  \cite[Prop. 5.3.4, Cor. 5.3.7, Thm. 6.1.6 e)]{NR}, for such $D > M$, one has
\begin{equation}\label{eq:NR1}
 \mathbb{H}^{\ast}(S^3_{-D}(K), j) \cong \mathbb{H}^{\ast}([0,m_1] \times \dots \times [0, m_{\nu}], w_j).
\end{equation}
Moreover, by  \cite[Prop. 7.1.3]{NR}, for $D>M$ the normalized Euler characteristic of this
cohomology can be compared with the coefficients of the polynomial $Q$. Namely,
\begin{equation}\label{eq:NR2}
 q_j = {\rm eu\ }\mathbb{H}^{\ast}([0,m_1] \times \dots \times [0, m_{\nu}], w_j).
\end{equation}
Then (\ref{eq:NR0}), (\ref{eq:NR1}) and (\ref{eq:NR2}) combined give
$q_j= - {\rm eu\ }\mathbb{H}^{\ast}\left(T_j, W\right)$ for any $j$.
Notice that $q_j=0$ if $j$ is not in the interval $[0,2\delta-2]$, and for these values by (\ref{eq:symQ}) one also has
$q_j=q_{2\delta-2-j}+\delta-j-1$, which equals $F(j)+\delta-j-1$ by (\ref{eq:QF}). Hence (\ref{eq:NR0})
(now applied with the original $d$) implies (\ref{eq:generalhq}).
\end{proof}

\begin{remark} \label{generalh0complete}
 In fact, the integer $d$, the sum of delta-invariants $\delta$ and the function $H$ completely determine the whole $\mathbb{H}^0$ as a graded $\mathbb{Z}[U]$-module (and not just its Euler characteristic). For this fact, we refer to \cite[Lemma 6.1.1, Theroem 6.1.6]{NR}.
\end{remark}


\begin{corollary}\label{cor:euformula} Assume that $d(d-3)=2\delta-2$,  cf. (\ref{eq:deltasum}). Then
 \[ {\rm eu\, } \mathbb{H}^0\left(S^3_{-d}(K),a\right) = \sum_{\substack{j \equiv -a (\textrm{mod\ } d)
 \\ 0\leq j \leq 2\delta-2}} H(j+1), \]
 \[ {\rm eu\, } \mathbb{H}^{\ast}\left(S^3_{-d}(K),a\right) = \sum_{\substack{j \equiv -a (\textrm{mod\ }
  d)\\ 0\leq j \leq 2\delta-2}} F(j). \]
The value $a = 0$ corresponds to the {\it canonical} {\rm Spin}$^c$--structure. Denote the corresponding lattice cohomology
of $S^3_{-d}(K)$ by  $\mathbb{H}^{\ast}_{\textrm{can}}(S^3_{-d}(K))$. Then the above identities read as:
\begin{equation}\label{eq:h0formula}
{\rm eu\, } \mathbb{H}^0_{{\rm can}}\left(S^3_{-d}(K)\right) =
\sum_{ 0\leq j \leq d-3}  H(jd+1),
\end{equation}
\begin{equation}\label{eq:hastformula}
{\rm eu\, } \mathbb{H}^{\ast}_{{\rm can}}\left(S^3_{-d}(K)\right) =
\sum_{ 0\leq j \leq d-3} F(jd).
\end{equation}
\end{corollary}
\begin{proof}
Use the symmetry properties (\ref{eq:symQ}) and (\ref{lem:symH}). \end{proof}

\subsection{Reformulations of Theorem \ref{thmbl} and Conjecture \ref{conj:blmn}} \

Because of the inequalities (implied by B\'ezout's theorem) of Lemma \ref{lem:Bezout},
Theorem \ref{thmbl} of Borodzik and Livingston is true if and only if the corresponding sums (over $j$) are equal.
This combined vith (\ref{eq:h0formula}) provides the following equivalent form.

\begin{theorem} {\bf (Alternative form of Theorem \ref{thmbl})}\label{thmblalt}
For a link $L=S^3_{-d}(K)$ of a superisolated surface singularity corresponding to a rational cuspidal projective plane
curve of degree $d$ we have:
\[ {\rm eu\ } \mathbb{H}^0_{{\rm can}}\left(S^3_{-d}(K)\right) = d(d-1)(d-2)/6. \]

\end{theorem}

This form is also present in the recent article \cite{NS} (in Example 2.4.3 (a) and Section 3).

Next, using (\ref{eq:hastformula}) we give an equivalent formulation of Conjecture \ref{conj:weak} in terms of lattice cohomology (cf. Conj. \ref{conj:weakalt}).

\begin{conjecture}\label{conj:weakalt2}  {\bf (Index theoretical version, second  alternative form)} \

For a link $L=S^3_{-d}(K)$ of a superisolated surface singularity corresponding to a
rational cuspidal projective plane curve of degree $d$ we have:
\[ {\rm eu\ } \mathbb{H}^{\ast}_{{\rm can}}\left(L\right) \leq \frac{d(d-1)(d-2)}{6}. \]
Alternatively, in the light of the previous theorem:
\[ {\rm eu\ } \mathbb{H}^{\ast}_{{\rm can}}\left(L\right) \leq{\rm eu\ } \mathbb{H}^0_{{\rm can}}\left(L\right).\]
\end{conjecture}

In this context, the weakened Conjecture \ref{conj:weakalt} is much more natural than the original
(and in general surely false) Conjecture \ref{conj:blmn} which would require the validity of $F(jd) \leq (j+1)(j+2)/2$
for every single $j = 0, 1, \dots, d-3$, i.e. an inequality for the lattice cohomological Euler characteristic of
each diagonal set $T_{jd}$.

\begin{remark}
In \cite{NJEMS} is proved that ${\rm eu\, } \mathbb{H}^{\ast}\left(S^3_{-d}(K),a\right) $
equals the normalized Seiberg--Witten invariant of $S^3_{-d}(K)$ associated with
the Spin$^c$--structure $a$. In particular, for the canonical Spin$^c$--structure, one has
$ {\rm eu\, } \mathbb{H}^{\ast}_{{\rm can}}\left(S^3_{-d}(K)\right)=
-\frsw_{\textrm{can}}(S^3_{-d}(K))-(K_{\textrm{can}}^2+s_{\widetilde{X}})/8$
(compatibly with the first version of Conjecture \ref{conj:weak}).
\end{remark}

\begin{remark}
 We wish to emphasize that it is essential that in the above conjecture we talk about the lattice cohomologies corresponding to the \emph{canonical} Spin$^c$--structure only. Using formulae of Corollary \ref{cor:euformula} one can check easily that for superisolated singularity link $L$ coming from the (existing) curve of Example \ref{ex:counter} choosing Spin$^c$--structure corresponding to $a = 4$ we have $\textrm{eu\ } \mathbb{H}^{\ast}(L,a=4) = 45 > \textrm{eu\ }\mathbb{H}^{0}(L,a=4) = 42$.
 Also, for many curves from series (1) in Proposition \ref{prop:tricuspidal} of the next section, one can find Spin$^c$--structures for which the inequality fails, e.g. $\textrm{eu\ } \mathbb{H}^{\ast}(L(C_{4,1}),a=2) = 3 > \textrm{eu\ }\mathbb{H}^{0}(L(C_{4,1}),a=2) = 2$ (cf. Example \ref{ex:first222}).
\end{remark}

\begin{remark}
 It is well-known that the link $L$ of a superisolated singularity corresponding to a projective plane curve is a rational homology sphere ($\mathbb{Q}HS^3$) if and only if the curve is rational and cuspidal.
 (See 7.1 in \cite{BLMN2} and the references therein.)

 The fact that $L$ is $\mathbb{Q}HS^3$ is probably also essential in the above conjecture. To see this, consider the following example. It is not hard to see by a construction using Cremona-transformations that there exists a rational projective curve $C$ of degree $d = 5$ with three singular points which are of the following type:
 One singularity is a simple transversal self-intersection (a reducible $A_1$-singularity). The other two are locally irreducible singularities, with multiplicity sequences $[3,2]$, resp. $[2]$ (alternatively, with Newton pairs $(3,5)$, resp. $(2,3)$). From the embedded resolution graphs of plane curve singularities, it is easy to construct the plumbing graph of the link $L$ of the corresponding superisolated singularity. (Due to the locally reducible singularity, it has one cycle, so it is not a tree.) One checks (presumably with the help of computer) that for this link $\textrm{eu\ } \mathbb{H}^{\ast}_{\textrm{can}}\left(L\right) = 11 > \textrm{eu\ } \mathbb{H}^0_{\textrm{can}}\left(L\right) = 10$.
 (Note that although $L$ is not a rational homology sphere, we can still speak about
 the corresponding lattice cohomologies with the same definition as in the $\mathbb{Q}HS^3$ case.)
\end{remark}

\subsection{Proof of Conjecture \ref{conj:weak} (index theoretical version, second alternative form \ref{conj:weakalt2}) for $\nu=2$.}
\label{ss:3.3}\

First we recall that ${\mathbb H}^q(S^3_{-d}(K),a)=0$ for any $q\geq \nu$. This follows
from the fact that the non-compact simplicial subcomplexes
$S_n$ of $\R^\nu$ (in the reduced lattices) have no nonzero homologies $H^q(S_n,\Z)$ for $q\geq \nu$;
or just apply  \cite{LN} or
\cite[6.2.1]{NT}.
Then, for $\nu=2$, we have ${\rm eu\ } {\mathbb H}^{\ast}(S^3_{-d}(K),a)= {\rm eu\ } {\mathbb H}^{0}(S^3_{-d}(K),a)-{\rm rank}_\Z\, {\mathbb H}^{1}(S^3_{-d}(K),a)$, hence the second alternative form transforms into ${\rm rank}_\Z\, {\mathbb H}^1_{{\rm can}}(S^3_{-d}(K))\geq 0$,
which is certainly true.

Notice also that for $\nu\geq 3$ similar argument does not work (and from this point of view,
it is even more surprising that in all the known cases,
the conjecture holds, cf. section \ref{s:Verify}).

\subsection{Proof of Conjecture \ref{conj:blmn} (alternative form \ref{conj:firstref}) for $\nu=2$.}
\label{ss:3.4}\

In fact, essentially by the same argument, in case of $\nu = 2$ one can prove the original, stronger Conjecture \ref{conj:firstref} as well. Using formula (\ref{eq:bor_liv}) of Theorem \ref{thmbl} and comparing it with (\ref{eq:u}), it is enough to prove that for $\nu = 2$ the inequality $F(k) \leq H(k+1)$ holds for \emph{any} $0 \leq k \leq 2\delta-2$. This inequality is purely combinatorial, completely independent of the parameter $d$, and has nothing to do with the realizability of cusp types on an existing rational projective curve (neither with the validity or failure of equalities (\ref{eq:bor_liv})).

In fact, similarly as in the proof of Theorem \ref{rem:generalformula}, set $D > 2\delta-2$. Then, by (\ref{eq:generalh0}) and (\ref{eq:generalhq}), the inequality $F(k) \leq H(k+1)$ turns into ${\rm eu\, } \mathbb{H}^{\ast}\left(S^3_{-D}(K),k\right) \leq {\rm eu\, } \mathbb{H}^0\left(S^3_{-D}(K),k\right)$ which is again true, since due to the vanishing following from the reduction principle, the difference is the only summand ${\rm rank}_\Z\, {\mathbb H}^1_{{\rm red}}(S^3_{-D}(K),k) \geq 0$.

\section{Verifying conjecture \ref{conj:weakalt2} for known curves with $\nu \geq 3$}\label{s:Verify}

\subsection{}

In this section we show that Conjecture \ref{conj:weakalt2} is true
for all the  rational cuspidal curves with at least three cusps currently known
(by the authors). For the list of such curves we refer to \cite[2.4.5]{Moe}, \cite[Conj. 4]{Pio}.
There are three infinite series of tricuspidal curves; one is a
two-parameter family, the other two series have one parameter each (the curve degree).

There are two `sporadic'
curves not contained in any of the three series. Both of them is of degree $5$; one is tricuspidal, the other has four cusps (this curve is conjectured to be the only rational cuspidal curve with more than three cusps).

The data for the curves with three cusps in the three infinite series are as follows (we present the multiplicity
sequences and, for the convenience, also the Newton pairs of the cusp types):

\begin{proposition}\label{prop:tricuspidal} \emph{(Flenner, Zaidenberg and Fenske; see \cite[3.5]{FZ1},
\cite[1.1]{FZ2}, \cite{Fen}, cf. also \cite[2.4.5]{Moe}, \cite{Pio})}

The following rational cuspidal curves exist:
\begin{enumerate}

	\item A curve $C_{d,u}$ (with $d \geq 4$ and $1 \leq u \leq d-3$) is of degree $d$ and has the following cusp types:
	  \subitem $[d-2]$, alternatively $(d-2,d-1)$
	  \subitem $[2_u]$, alternatively $(2, 2u+1)$
	  \subitem $[2_{d-2-u}]$, alternatively $(2, 2d-2u-3)$.
	
	  (Here it is enough to take $u \leq \left\lfloor \frac{d-2}{2}\right\rfloor$ or $\left\lceil \frac{d-2}{2}\right\rceil \leq u$,
	  as $C_{d,u} = C_{d,d-2-u}$.)
	
	\item A curve $D_{l}$ (with $l \geq 1$) is of degree $d = 2l + 3$ and has the following cusp types:
	  \subitem $[2l, 2_l]$, alternatively $(l,l+1)(2,1)$ (case $l=1$ degenerates to one Newton pair)
	  \subitem $[3_l]$, alternatively $(3,3l+1)$
	  \subitem $[2]$, alternatively $(2,3)$.
	
	\item A curve $E_{l}$ (with $l \geq 1$) is of degree $d = 3l + 4$ and has the following cusp types:
	  \subitem $[3l, 3_l]$, alternatively $(l,l+1)(3,1)$ (case $l=1$ degenerates to one Newton pair)
	  \subitem $[4_l, 2_2]$, alternatively $(2,2l+1)(2,1)$
	  \subitem $[2]$, alternatively $(2,3)$.
\end{enumerate}

\end{proposition}

\begin{theorem}

Let us denote the links of superisolated singularities corresponding to curves
$C_{d,u}$ (resp. $D_l$, $E_l$) by $L(C_{d,u})$ (resp. $L(D_l)$, $L(E_l)$). Then we have:
\begin{enumerate}
 \item $\textrm{eu\ } \mathbb{H}_{\textrm{can}}^{0}(L(C_{d,u})) - \textrm{eu\ }
\mathbb{H}^{\ast}_{\textrm{can}}(L(C_{d,u})) = \begin{cases} l(l-1) &\mbox{if }
d = 2l+1 \\ (u-l)(u-l+1) & \mbox{if } d = 2l, u \geq l-1. \end{cases} $

\vspace{2mm}

\noindent In particular, $\textrm{eu\ } \mathbb{H}_{\textrm{can}}^{0}(L(C_{d,u})) -
\textrm{eu\ }\mathbb{H}^{\ast}_{\textrm{can}}(L(C_{d,u})) \geq 0$ in all cases.

\vspace{2mm}

 \item $\textrm{eu\ } \mathbb{H}_{\textrm{can}}^{0}(L(D_l)) - \textrm{eu\ }\mathbb{H}^{\ast}_{\textrm{can}}(L(D_l)) =
       \begin{cases} 4p(3p-1)+2 & \mbox{if } l = 3p-1 \\
                     4p(3p-1) & \mbox{if } l = 3p \\
                     12p(p+1)+2 &\mbox{if } l = 3p+1
       \end{cases}$

\vspace{2mm}

  \item $\textrm{eu\ } \mathbb{H}_{\textrm{can}}^{0}(L(E_l)) - \textrm{eu\ }\mathbb{H}^{\ast}_{\textrm{can}}(L(E_l)) =
        \begin{cases} 60p^2-2p &\mbox{if } l = 4p \\
                      60p^2+46p+10 & \mbox{if } l = 4p+1 \\
                      60p^2+62p+16 & \mbox{if } l = 4p+2 \\
                      60p^2+100p+42 & \mbox{if } l = 4p+3
        \end{cases}$

\vspace{2mm}

\noindent In particular, Conjecture \ref{conj:weakalt2} is satisfied in each case.
\end{enumerate}

\end{theorem}

\begin{proof}
 Since in each case we know explicitly the three singularities, we also know explicitly the product $\Delta(t)$
of their Alexander-polynomials. Therefore, it is convenient to work with formula (\ref{eq:R}), as (\ref{eq:R2}) reads as:
 \[ R(1) = \textrm{eu\ } \mathbb{H}^{\ast}_{\textrm{can}} - \textrm{eu\ } \mathbb{H}^{0}_{\textrm{can}}. \]
 We do not give the computations here, just present the form of the polynomial $\Delta(t) = \Delta_1(t)\Delta_2(t)\Delta_3(t)$
in terms of the parameters in each case.
 \begin{enumerate}
  \item \[\Delta^{(C)}(t) = \frac{(t-1)(t^{(d-2)(d-1)}-1)}{(t^{d-2}-1)(t^{d-1}-1)} \frac{(t-1)(t^{2(2u+1)}-1)}{(t^{2}-1)(t^{2u+1}-1)}
 \frac{(t-1)(t^{2(2d-2u-3)}-1)}{(t^{2}-1)(t^{2d-2u-3}-1)} \]
  \item \[\Delta^{(D)}(t) = \frac{(t-1)(t^{2l(l+1)}-1)(t^{2+4l(l+1)}-1)}{(t^{2l}-1)(t^{2(l+1)}-1)(t^{1+2l(l+1)}-1)}
  \frac{(t-1)(t^{3(3l+1)}-1)}{(t^{3}-1)(t^{3l+1}-1)} \frac{(t-1)(t^{6}-1)}{(t^{2}-1)(t^{3}-1)} \]
  \item \[\Delta^{(E)}(t) = \frac{(t-1)(t^{3l(l+1)}-1)(t^{3+9l(l+1)}-1)}{(t^{3l}-1)(t^{3(l+1)}-1)(t^{1+3l(l+1)}-1)}
  \frac{(t-1)(t^{4(2l+1)}-1)(t^{2+8(2l+1)}-1)}{(t^{4}-1)(t^{2(2l+1)}-1)(t^{1+4(2l+1)}-1)}
  \frac{(t-1)(t^{6}-1)}{(t^{2}-1)(t^{3}-1)} \]
 \end{enumerate}

 Then, in each case, use formula (\ref{eq:R}) with the corresponding $d$ to obtain the result.
\end{proof}

\begin{remark} \label{rem:detailed}
We present a table containing the detailed data $H(jd+1) - F(jd)$, $j = 0, \dots, d-3$ for the first few members of the first series:

\vspace{2mm}

\begin{center}
\begin{tabular}{| c | c | c | c | c | c | c | c | c | c | c |} \hline
  Curve     & Degree & Cusp types & \multicolumn{7}{c|}{$H(jd+1) - F(jd)$} & eu $\mathbb{H}^{0} - $ eu $\mathbb{H}^{\ast}$ \\ \hline \hline
  $C_{4,1}$ & $d = 4$ & $[2]$, $[2]$, $[2]$ & $0$ & $0$ & & & & & & $0$ \\ \hline \hline
  $C_{5,1}$ & $d = 5$ & $[3]$, $[2_2]$, $[2]$ & $0$ & $2$ & $0$ & & & & & $2$ \\ \hline \hline
  $C_{6,1}$ & $d = 6$ & $[4]$, $[2_3]$, $[2]$ & $0$ & $0$ & $0$ & $0$ & & & & $0$ \\ \hline
  $C_{6,2}$ & $d = 6$ & $[4]$, $[2_2]$, $[2_2]$ & $0$ & $0$ & $0$ & $0$ & & &  & $0$ \\ \hline \hline
  $C_{7,1}$ & $d = 7$ & $[5]$, $[2_4]$, $[2]$ & $0$ & $3$ & $0$ & $3$ & $0$ & &  & $6$ \\ \hline
  $C_{7,2}$ & $d = 7$ & $[5]$, $[2_3]$, $[2_2]$ & $0$ & $3$ & $0$ & $3$ & $0$ & &  & $6$ \\ \hline \hline
  $C_{8,1}$ & $d = 8$ & $[6]$, $[2_5]$, $[2]$ & $0$ & $0$ & $1$ & $1$ & $0$ & $0$ &  & $2$ \\ \hline
  $C_{8,2}$ & $d = 8$ & $[6]$, $[2_4]$, $[2_2]$ & $0$ & $-1$ & $1$ & $1$ & $-1$ & $0$ &  & $0$ \\ \hline
  $C_{8,3}$ & $d = 8$ & $[6]$, $[2_3]$, $[2_3]$ & $0$ & $-1$ & $1$ & $1$ & $-1$ & $0$ &  & $0$ \\ \hline \hline
  $C_{9,1}$ & $d = 9$ & $[7]$, $[2_6]$, $[2]$ & $0$ & $3$ & $1$ & $4$ & $1$ & $3$ & $0$  & $12$ \\ \hline
  $C_{9,2}$ & $d = 9$ & $[7]$, $[2_5]$, $[2_2]$ & $0$ & $4$ & $0$ & $4$ & $0$ & $4$ & $0$  & $12$ \\ \hline
  $C_{9,3}$ & $d = 9$ & $[7]$, $[2_4]$, $[2_3]$ & $0$ & $4$ & $0$ & $4$ & $0$ & $4$ & $0$  & $12$ \\ \hline

\end{tabular}
\end{center}

\vspace{2mm}

We see that the smallest degree where the original Conjecture \ref{conj:blmn} fails is degree $8$.
The general pattern for larger $d$'s is that the conjecture fails only at even degrees and when $d - 3 > u > 1$
(so it is still true when the degree is odd or the degree is even and $u = 1$ or $u = d-3$).

Computations show that the other two series satisfy even the original Conjecture \ref{conj:blmn}.

\end{remark}

\bekezdes \label{checksporadic}
Finally, check the conjecture for the two exceptional curves. The tricuspidal one of degree $d = 5$ has cusp types $[2_2], [2_2], [2_2]$. The numerical values of functions $F$ and $H$ are as follows:

\vspace{2mm}

\begin{center}
\begin{tabular}{|c | c | c | c | c | c | c | c | c | c | c | c |} \hline
  $k$ & $0$ & $1$ & $2$ & $3$ & $4$ & $5$ & $6$ & $7$ & $8$ & $9$ & $10$ \\ \hline
  $H(k+1)$ & $1$ & $1$ & $2$ & $2$ & $3$ & $3$  & $4$  & $4$  & $5$  & $5$  & $6$ \\ \hline
  $F(k)$ & $1$ & $-1$ & $3$ & $-3$ & $6$ & $-3$  & $7$  & $-1$  & $6$  & $3$  & $6$  \\ \hline
  $H(k+1)-F(k)$ & $0$ & $2$ & $-1$ & $5$ & $-3$ & $6$  & $-3$  & $5$  & $-1$  & $2$  & $0$  \\ \hline

\end{tabular}
\end{center}

\vspace{2mm}

Hence $\textrm{eu\ }
\mathbb{H}^{0}_{\textrm{can}} - \textrm{eu\ } \mathbb{H}^{\ast}_{\textrm{can}} = 0+6+0 > 0$.

The single known rational cuspidal curve with four cusps has
degree $d = 5$. The cusp types are $[2_3], [2], [2], [2]$. The detailed data:

\vspace{2mm}

\begin{center}
\begin{tabular}{|c | c | c | c | c | c | c | c | c | c | c | c |} \hline
  $k$ & $0$ & $1$ & $2$ & $3$ & $4$ & $5$ & $6$ & $7$ & $8$ & $9$ & $10$ \\ \hline
  $H(k+1)$ & $1$ & $1$ & $2$ & $2$ & $3$ & $3$  & $4$  & $4$  & $5$  & $5$  & $6$ \\ \hline
  $F(k)$ & $1$ & $-2$ & $5$ & $-5$ & $8$ & $-5$  & $9$  & $-3$  & $8$  & $2$  & $6$  \\ \hline
  $H(k+1)-F(k)$ & $0$ & $3$ & $-3$ & $7$ & $-5$ & $8$  & $-5$  & $7$  & $-3$  & $3$  & $0$  \\ \hline

\end{tabular}
\end{center}

\vspace{2mm}

Hence $\textrm{eu\ }
\mathbb{H}^{0}_{\textrm{can}} - \textrm{eu\ } \mathbb{H}^{\ast}_{\textrm{can}} = 0+8+0 > 0$.

\section{A combinatorial surgery formula for  $\mathbb{H}^0(S^3_{-d}(K))$ and
$\textrm{eu\ }\mathbb{H}^0(S^3_{-d}(K))$}\labelpar{s:hilbertf}

\subsection{Dependence of $\mathbb{H}^0(S^3_{-d}(K))$
and $\textrm{eu\ }\mathbb{H}^0(S^3_{-d}(K))$ on the multiplicity sequences}
In this section we present an effective way to compute $\textrm{eu\ }\mathbb{H}^0(S^3_{-d}(K), a)$,
where the setting is as in Section \ref{s:Lattice} and in \cite{NR}, i.e. $K=\#_{i=1}^\nu
K_i$, $d$ is an arbitrary positive integer, and
 $a$ stands for a Spin$^c$--structure, hence  $a \in\{ 0, \dots, d-1\}$. In this discussion we prefer to
 fix the integers $d$, $a$ and $\delta$. Hence, by (\ref{eq:generalh0}),
 $\textrm{eu\ }\mathbb{H}^0(S^3_{-d}(K))$ (and, by Remark \ref{generalh0complete} the $\mathbb{Z}[U]$-module $\mathbb{H}^0(S^3_{-d}(K))$ as well) is completely determined by the minimum--convolution
 $H= H_1\diamond\cdots\diamond  H_\nu $.
 In this section we focus on the dependence of $H$ on the
 multiplicity sequences of plane curve singularities corresponding to knots $K_i$.

\bekezdes
Let $[n^{(1)}_1, \dots, n^{(1)}_{r_1}], \dots, [n^{(\nu)}_1, \dots, n^{(\nu)}_{r_{\nu}}]$
be the multiplicity sequences of the local singularities.
(We omit the 1's at the end of the multiplicity sequences, i.e. we define the multiplicity sequence
 as a sequence of multiplicities occurring in consecutive blowups resulting in \emph{smooth}
 exceptional divisors and strict transform, but not necessarily in normal crossings of exceptional divisors and
 strict transform; in particular, $n^{(i)}_{r_{i}} > 1$ for all $i = 1, \dots, \nu$.)

Then the sum of delta invariants of the singularities is
\begin{equation}\label{eq:deltamult}
\delta = \sum_{i = 1}^{\nu} \sum_{j = 1}^{r_i} \frac{n^{(i)}_j(n^{(i)}_j-1)}{2}.
\end{equation}
We will show that the minimum-convolution $H = H_1 \diamond \dots \diamond H_{\nu}$
depends only on the {\it multiset} of multiplicities
$\{ \{ n^{(1)}_1, \dots, n^{(1)}_{r_1}, \dots, n^{(\nu)}_1, \dots n^{(\nu)}_{r_{\nu}} \} \}$.
(By a multiset we mean a set, where the same element might be repeated and we keep track the number of
appearances; hence a multiset with integer entries basically is an element of the group ring $\Z[\Z]$.)

\begin{theorem}\label{thm:minconv}

 Assume we are given two collections of plane curve singularity types with their multiplicity sequences:
 \[ [n^{(1)}_1, \dots, n^{(1)}_{r_1}], \dots, [n^{(\nu)}_1, \dots, n^{(\nu)}_{r_{\nu}}] \]
 and
 \[ [\overline{n}^{(1)}_1, \dots, \overline{n}^{(1)}_{\overline{r}_1}], \dots, [\overline{n}^{(\overline{\nu})}_1, \dots,
 \overline{n}^{(\overline{\nu})}_{\overline{r}_{\overline{\nu}}}]. \]
Denote the counting functions of their semigroups by
 $H_1, \dots, H_{\nu}$ and $\overline{H}_1, \dots, \overline{H}_{\overline{\nu}}$, respectively.

 If
 \[ \{ \{ n^{(1)}_1, \dots, n^{(1)}_{r_1}, \dots, n^{(\nu)}_1, \dots, n^{(\nu)}_{r_{\nu}} \} \} =
     \{ \{ \overline{n}^{(1)}_1, \dots, \overline{n}^{(1)}_{\overline{r}_1}, \dots, \overline{n}^{(\overline{\nu})}_1, \dots, \overline{n}^{(\overline{\nu})}_{\overline{r}_{\overline{\nu}}} \} \} \]
 \emph{as multisets},
 then
 \[ H_1 \diamond \dots \diamond H_{\nu} = \overline{H}_1 \diamond \dots \diamond \overline{H}_{\overline{\nu}}. \]
\end{theorem}

\begin{proof}
 Denote the counting function of the semigroup of a singularity with mutliplicity sequence
 $[n_1, n_2, \dots, n_r]$ by $H_{[n_1, n_2, \dots, n_r]}$. Due to the obvious associativity of
 the minimum-convolution, for the statement of the theorem it is enough to show that
 \[ H_{[n_1, n_2, \dots, n_r]} = H_{[n_1]} \diamond  H_{[n_2, \dots, n_r]}. \]
 This will be proved in the next subsection (as Proposition \ref{prop:blowuphilbert}).
\end{proof}

\begin{corollary}\label{cor:combstab}
 Assume that $K = K_1 \# \dots \# K_{\nu}$ and $\overline{K} = \overline{K}_1 \# \dots \# \overline{K}_{\overline{\nu}}$ are connected sums of algebraic knots with summands as above.
 If the collections of numbers coming from multiplicity sequences corresponding to the algebraic knots are equal
 as multisets, then
 \[ \mathbb{H}^0(S^3_{-d}(K),a) \cong \mathbb{H}^0(S^3_{-d}(\overline{K}),a) \]
 for any integer $d > 0$ and any {\rm Spin}$^c$--structure $a \in\{ 0, \dots, d-1\}$.

 The same is true for
 ${\rm eu}\,\mathbb{H}^0$ as well.
\end{corollary}
\begin{proof} Use (\ref{eq:deltamult}) and  Remark \ref{generalh0complete} (resp. formula (\ref{eq:generalh0})).
\end{proof}

\begin{remark} \label{ren:Notstable}
 Note that a similar statement is not true for $\mathbb{H}^q$ ($q\geq 1$), not
even for the numerical value ${\rm eu\ } \mathbb{H}^{\ast}$;
 see e.g.
 %
%
%
%
the links of superisolated singularities corresponding to the curves from  \ref{checksporadic}.
 The two curves have the same degree, the same multiset of multiplicities,
 but different $F$-functions and different (canonical) lattice cohomologies as well.
\end{remark}

\bekezdes
It is quite surprising that from the point of view of the zeroth lattice cohomology only the collection
of multiplicities `put together' is important. This fact makes a lot easier to compute
$\mathbb{H}^0(S^3_{-d}(K),a)$ in many cases. We can view this result in the following way as well:
the zeroth lattice cohomology shows a stability with respect to the `combinatorial surgery' of
moving multiplicity numbers from one multiplicity sequence to another. We illustrate this by a
simple example from \cite[table in Section 2.3]{BLMN1}.

There exist cuspidal curves of degree $5$ with the following cusp data (we present the multiplicity sequences):
\begin{itemize}
 \item $[3], [2_3] \ (\nu=2)$
 \item $[3,2], [2_2] \ (\nu = 2)$
 \item $[3], [2_2], [2] \ (\nu = 3)$
\end{itemize}
Of course the corresponding surgery manifolds have $\mathbb{H}^0(S^3_{-d}(K),0)$ as prescribed by
Theorem \ref{thmblalt}. From Corollary \ref{cor:combstab} we see immediately without any computation
that these manifolds must have identical zeroth lattice cohomology and not only for the canonical
Spin$^c$--structure $a = 0$, but for the other values $a = 1, \dots, 4$ as well.

Notice that curves of degree $5$ with the following cusp data do not exist:
\begin{itemize}
 \item $[3,2], [2], [2] \ (\nu = 3)$
 \item $[3], [2], [2], [2] \ (\nu = 4)$
\end{itemize}
However, the corresponding surgery manifolds also have $\mathbb{H}^0(S^3_{-d}(K),a)$ as above
(and to see this we do not need any further computations, since it is obvious from the multiset
of multiplicities).

In general, we have the following statement.
\begin{corollary} \label{cor:testweak}
 Assume that we have local singularity types which are candidates to be the singularities of a
 rational cuspidal plane curve of degree $d$ in the sense of Definition \ref{def:candidates},
 and they satisfy the necessary condition given by Theorem \ref{thmbl}.
 Then any other collection of local singularity types, such that the multiset of the occurring
 multiplicities is the same as in the case of the original collection of singularities,
 as a new candidate  satisfies the necessary condition given by Theorem \ref{thmbl} as well.
\end{corollary}
\begin{remark} \label{rem:stability}

(a) This result enlarge the applicability of the criterion provided by Borodzik--Livingston Theorem \ref{thmbl}
drastically: if we reorganize the multiplicity numbers of a
rational cuspidal curve candidate (by keeping the multiset), then the
 new combinatorial candidate satisfies the
 output of the Borodzik--Livingston Theorem \ref{thmbl}
 if and only if the original candidate satisfied it (regardless of the algebraic realizability).

 This shows that although in the  Borodzik--Livingston Theorem \ref{thmbl}
 the algebraic realizability is important, in reality it matters `less', and presumably
 it can be replaced by a much weaker assumption. E.g., a possible `assumption candidate' requires only the
 smooth realizability of the curve (near the singular points a smooth model of the singular local
 embeddings, otherwise a smooth embedding). In fact, analyzing the proof of \cite{BL1},
 only this data is used. It would be interesting to prove that two
 candidates with equivalent data in the sense of Theorem \ref{thm:minconv} can/cannot  be
 simultaneously smoothly embedded in ${\mathbb C}{\mathbb P}^2$.

(b) As there exist rational cuspidal curves with arbitrarily long multiplicity sequences (even in
the unicuspidal case, see e.g. Orevkov's curves in \cite{Or}), 
the above corollary also shows that Theorem \ref{thmbl} cannot provide
any restriction on the number of cusps of rational cuspidal curves. (It is conjectured that the number
of cusps is always less than five, i.e. $\nu \leq 4$, see e.g. \cite{Pio}. A result of Tono shows that
$\nu \leq 8$, see \cite{Tono}, cf. also Example 6.16 of \cite{BL1}.)
\end{remark}
\subsection{The behaviour of the counting function under the blowup.}

The goal of this subsection is to prove Proposition \ref{prop:blowuphilbert}, thus completing the proof
of Theorem \ref{thm:minconv}. The notations in this subsection are completely independent of the other
parts of this article.

Let $\Gamma_1$ and $\Gamma_2$ be semigroups of plane curve singularities. We will assume that
 $\Gamma_1$ is the blowup of $\Gamma_2$ (once).
 (That is, the first cusp is obtained from the second by blow up.)
 Let $H_1(i)$ and $H_2(i)$ be the corresponding counting functions,
  i.e. $H_\ell(i) = \#\{ s \in \Gamma_\ell : s < i \}$. Our goal is to compare $H_1$ and $H_2$.

Denote by $m$ the \emph{multiplicity} of the second (`more complex') singularity, i.e.
$m = \textrm{min}\{s \in \Gamma_2 : 0 < s\}$. The \emph{Ap\'{e}ry set} of a numerical semigroup with
respect to one of its elements is a standard invariant commonly used in semigroup theory.
It consists of the smallest elements of the semigroup from each (nonempty) residue class modulo the
given element. We consider the Ap\'{e}ry set of $\Gamma_2$ with respect to
$m$, that is,
 $\textrm{Ap}(m, \Gamma_2) = \{b_0, b_1, \dots, b_{m-1}\}$, where $0 = b_0 < b_1 < \dots < b_{m-1}$.
It is a complete residual system modulo $m$, and by the definition, for each $i \ (0 \leq i \leq m-1)$ we have
$b_i \in \Gamma_2$ but $b_i - m \notin \Gamma_2$.

The definition guarantees that $\Gamma_2 = \textrm{Ap}(m, \Gamma_2) + m \cdot \mathbb{Z}_{\geq 0}$.
In fact, for every element $s \in \Gamma_2$ there exist \emph{uniquely} $j \in \mathbb{Z}$ and
$u \in \mathbb{Z}$ such that $s = b_j + mu$, $0 \leq j \leq m-1$, $0 \leq u$.

\begin{lemma}
 If $\Gamma_1$ is the blowup of $\Gamma_2$, then $m \in \Gamma_1$ as well.
\end{lemma}
\begin{proof}
The strict transform of the singular curve after the blowup and the reduced exceptional divisor of the blowup have intersection multiplicity $m$.
\end{proof}

 Therefore, we can consider the Ap\'{e}ry set with respect to $m$ of $\Gamma_1$ as well:
 $\textrm{Ap}(m, \Gamma_1) = \{ a_0, a_1, \dots, a_{m-1} \}$ is a complete residual system mod
 $m$ such that $a_i \in \Gamma_1$ but $a_i - m \notin \Gamma_1$ for all $0 \leq i \leq m-1$, and
 $0 = a_0 < a_1 < \dots < a_{m-1}$. Again, $\Gamma_1 = \textrm{Ap}(m, \Gamma_1) +
  m \cdot \mathbb{Z}_{\geq 0}$, i.e. for any $ s \in \Gamma_1$ there exist  unique
  $ j, u \in \mathbb{Z}$  such that
 $  s = a_j + mu, \ 0 \leq j \leq m-1, \ 0 \leq u$.

\begin{proposition}\label{prop:AB} (\cite[Lemme 2]{Ap}, \cite[Prop. 2.3]{BDF})
 The blowup of a semigroup is described by the two  Ap\'{e}ry sets with respect to the original
 multiplicity $m$ in the following way:
 \[ a_j + jm = b_j \ \ \ \ (j=0,\ldots,m-1). \]
\end{proposition}

\begin{rem}
The previous proposition  implies that `the order is preserved', i.e. if $\Gamma_2$ is a semigroup
of a plane curve singularity which has multiplicity $m$ and the \emph{ordered} Ap\'{e}ry set with
respect to this multiplicity is $\textrm{Ap}(m, \Gamma_2) = \{b_0, b_1, \dots, b_{m-1}\}$ with
$0 = b_0 < b_1 < \dots < b_{m-1}$, then the series of inequalities
$0 = b_0 - 0\cdot m < b_1 - m < b_2 - 2m < \dots < b_{m-1} - (m-1)m$ must be satisfied.
(This is a nontrivial necessary condition for an algebraic numerical semigroup to be a semigroup
of a plane curve singularity!)
\end{rem}

Let $\Gamma_{[m]}$ be the semigroup of the plane curve singularity with multiplicity sequence $[m]$,
it is generated as a semigroup by $m$ and $m+1$. Denote
its counting function by $H_{[m]}$.

The counting functions $H_1$, $H_2$ and $H_{[m]}$ are related as follows.

\begin{prop}\label{prop:blowuphilbert} For all $l \geq 0$ one has
 \[ H_2(l) = \min_{0 \leq j \leq l}\  \{ H_1(l-j) + H_{[m]}(j) \}. \]
\end{prop}

\begin{proof}
 We need to prove that for all $l \geq 0$, and for all $ j$ with $ 0 \leq j \leq l$ one has
 $ H_2(l) - H_1(l-j) \leq H_{[m]}(j)$; furthermore, that for all
 $l \geq 0$ equality holds for some $j$.

 It will be useful to view the semigroups as unions of `layers' according to the Ap\'{e}ry sets.
 Namely, for $i = 0, 1, \dots, m-1$ set $\Gamma_1^{(i)} := a_i + m\mathbb{Z}_{\geq 0}$.
 Then $\Gamma_1 = \sqcup_{i=0}^{m-1}\Gamma_1^{(i)}$ as a \emph{disjoint} union.
 Similarly, set $\Gamma_2^{(i)} := b_i + m\mathbb{Z}_{\geq 0}$, hence
$\Gamma_2 =\sqcup_{i=0}^{m-1} \Gamma_2^{(i)}$ as a \emph{disjoint} union.  Then
\[ H_1(l-j) = \sum_{i}\ \#\{ s \in \Gamma_1^{(i)} : s < l-j \}, \ \ \
H_2(l) = \sum_i\ \#\{ s \in \Gamma_2^{(i)} : s < l \}. \]
By Proposition \ref{prop:AB} we get that the $i$th layer $\Gamma^{(i)}_2$ of the semigroup $\Gamma_2$ just has to be
shifted to the left by $im$ to get the $i$th layer $\Gamma^{(i)}_1$ of the semigroup $\Gamma_1$. Hence,
\[ H_2(l) = \sum_i\ \#\{ s \in \Gamma_1^{(i)} : s < l-im \}. \]
Now for a fixed $l$ the difference which has to be (sharply)
bounded from above  can be written as a difference of set--cardinalities,
the sets being differences of subsets of the semigroup (layers of) $\Gamma_1$:
\[ H_2(l) - H_1(l-j) = \#\{A_{j,l}\} - \#\{B_{j,l}\}, \]
where $A_{j,l} = \sqcup_{i=0}^{m-1}A_{j,l}^{(i)}$ and $B_{j,l} =\sqcup_{i=0}^{m-1}
 B_{j,l}^{(i)}$ as disjoint unions, with
\[ A_{j,l}^{(i)} = \{ s \in \Gamma_1^{(i)} : l-j \leq s < l-im  \}, \ \ \
 B_{j,l}^{(i)} = \{ s \in \Gamma_1^{(i)} : l-im \leq s < l-j  \}. \]
(Note that for all $i$, at least one of $A_{j,l}^{(i)}$ and $B_{j,l}^{(i)}$ is empty.)

Hence, we need  to prove that $\#\{A_{j,l}\} - \#\{B_{j,l}\} \leq H_{[m]}(j)$, and
for each $l$ equality holds for some $j = 0, 1, \dots, l$.

The inequality follows from  $\#\{B_{j,l}\} \geq 0$ and  $\#\{A_{j,l}\} \leq H_{[m]}(j)$.
This second inequality is not straightforward. First we check it for the multiples of $m$, i.e.
for $j$'s of form $j = \mu m$:
\begin{equation}\label{eq:HM}
\#\{ A_{\mu m, l} \} = \sum_{i=0}^{m-1} \#\{ A_{\mu m, l}^{(i)} \} \leq \sum_{i=0}^{m-1}\textrm{max}\{ \mu - i, 0 \} = \sum_{i=0}^{\mu}\textrm{min}\{i, m\} = H_{[m]}(\mu m). \end{equation}
This is true because $\#\{ A_{\mu m, l}^{(i)} \} = \# \{ r \in \mathbb{Z}_{\geq 0} :
 l-\mu m \leq a_i + rm < l - im \} \leq \textrm{max}\{ \mu - i, 0 \}$. 
 (This upper bound is valid even if $\mu m > l$.)

For $j$'s not of the form $\mu m$ the inequality follows from two facts. 
First,  $0 \leq \#\{ A_{j+1, l} \} - \#\{ A_{j, l} \} \leq 1$ 
(for $\#\{ A_{j+1, l}^{(i)} \} - \#\{ A_{j, l}^{(i)} \} \in \{0, 1\}$, 
and except for at most one $i$, the difference is $0$, as elements of $A_{j, l}^{(i)}$ 
for different $i$'s have different residues modulo $m$).
 Second,  $H_{[m]}(j)$ as an upper bound,
 is `as generous as possible' between multiples of $m$, meaning that it increases from $H_{[m]}(\mu m)$ 
 to $H_{[m]}((\mu+1) m)$ `as fast as possible':  for $\mu m \leq j = \mu m + \gamma < (\mu + 1) m$
 being $H_{[m]}(\mu m + \gamma) = \textrm{min}\{\gamma, \mu + 1\} + H_{[m]}(\mu m)$.
These two facts and (\ref{eq:HM}) show the inequality for $0\leq \gamma\leq \mu+1$ 
(where $j=\mu m+\gamma$).
If $\mu+1<\gamma<m$, then we have 
$H_{[m]}(\mu m + \gamma) = H_{[m]}((\mu+1)m) \geq \#\{ A_{(\mu+1)m, l} \} \geq \#\{ A_{\mu m + \gamma, l} \}$
(here we have used the inequality for $(\mu + 1)m$).

Next we show that for any $l$ there exists a $j$ for which equality holds. From the above, it is clear 
which conditions do we want to be satisfied. We will choose a $j$ such that $0 \leq j \leq l$, 
$j \leq (m-1)m$, $B_{j,l} = \emptyset$ and $\# \{ r \in \mathbb{Z}_{\geq 0} : l-j \leq a_i + rm < l - im \}
 = \#\{ A_{j, l}^{(i)} \} = \textrm{max}\{ \lceil\frac{j}{m}\rceil - i, 0 \}$ for all $i = 0, 1, \dots, m-1$.

For any $l$, let $i_0$ be the smallest index $i$ among $0, 1, \dots, m-1$ for which $l - im \leq a_i$ is already valid, if such index exists. If not, take $i_0 = m-1$.

It is not hard to see that $j = \textrm{min}\{i_0m, l\}$ is a good choice, i.e. for $j = \textrm{min}\{i_0m, l\}$ we will have equality in the upper bound. For if $j = i_0m$, then by the choice of $i_0$ we have 
$B_{i_0m,l} = \emptyset$ and $A^{(i)}_{i_0m,l} = \textrm{max}\{ i_0 - i, 0 \}$ for all possible $i$, and these two conditions (via (\ref{eq:HM})) are enough to guarantee
the equality $\#\{A_{i_0m,l}\} - \#\{B_{i_0m,l}\} = H_{[m]}(i_0m)$. If $j = l$ happens to be the case (i.e. if $l < i_0m$),
then by a similar argument as above, we have $H_{[m]}(i_0m) = A_{i_0m,l}$, hence
$A_{l,l} \leq H_{[m]}(l) \leq H_{[m]}(i_0m) = A_{i_0m,l} = A_{l,l} $, which implies $\#\{A_{l,l}\} - \#\{B_{l,l}\} = H_{[m]}(l)$ again, as $B_{l,l} = \emptyset$. 
\end{proof}



\begin{thebibliography}{9}

\bibitem{Ap}
Ap\'{e}ry, R.: Sur les branches superlin\'{e}aires des courbes alg\'{e}briques,
\emph{C.R. Acad. Sci. Paris}, \textbf{222} (1946), 1198--1200.

\bibitem{BDF}
Barucci, V.; D'Anna, M. and Fr\"{o}berg, R.: On plane algebroid curves,
\emph{Lecture Notes Pure Appl. Math.}, \textbf{231} (2003), 37--50.

\bibitem{BLMN1}
  de Bobadilla, J. F.; Luengo, I.; Melle-Hernandez, A. and N\'{e}methi, A.: On rational cuspidal projective plane curves,
   \emph{Proc. London Math. Soc.}, \textbf{92} (1) (2006), 99--138.


\bibitem{BLMN2}
  de Bobadilla, J. F.; Luengo, I.; Melle-Hernandez, A. and N\'{e}methi, A.: On rational cuspidal curves, 
open surfaces and local singularities,
  \emph{Singularity theory, Dedicated to Jean-Paul Brasselet on His 60th Birthday, Proceedings of the
 2005 Marseille Singularity School and Conference}, (2007), 411--442.


\bibitem{BL1}
  Borodzik, M. and Livingston, Ch.: Heegaard--Floer homologies and rational cuspidal curves,
  \emph{Preprint}, arXiv:1304.1062, (2013).


\bibitem{NB} Braun, G. and N\'emethi, A.: Surgery formula for the Seiberg--Witten invariants of negative definite plumbed $3$-manifolds,
  \emph{Journal f\"ur die Reine und Angewandte Mathematik},  \textbf{638} (2010), 189--208.

\bibitem{BK} Brieskorn, E. and Kn\"orrer, H.: Plane Albegraic Curves,  \emph{Birkh\"auser Verlag}, (1986).

\bibitem{EN} Eisenbud, D. and Neumann, W.: Three-dimensional Link Theory and Invariants of Plane Curve Singularities,
  \emph{Ann. of Math. Studies 110, Princeton University Press, Princeton}, (1985).

\bibitem{Fen}
  Fenske, T.: Rational cuspidal plane curves of type $(d, d-4)$ with $\chi \leq 0$,
  \emph{Manuscripta Mathematica},  \textbf{98} (4) (1999), 511--527.

\bibitem{FZ1}
  Flenner, H. and Zaidenberg, M.: On a class of rational cuspidal plane curves,
   \emph{Manuscripta Mathematica},  \textbf{89} (1) (1996), 439--459.

\bibitem{FZ2}
  Flenner, H. and Zaidenberg, M.: Rational Cuspidal Plane Curves of Type $(d, d-3)$,
  \emph{Mathematische Nachrichten},  \textbf{210} (1) (2000), 93--110.

\bibitem{GDC}
  Gusein-Zade, S. M.; Delgado, F. and Campillo, A.: On the monodromy of a plane curve singularity and 
the Poincar\'{e} series of the ring of functions on the curve,
  \emph{Functional Analysis and Its Applications},  \textbf{33} (1) (1999), 56--57.

\bibitem{LN}
  L\'{a}szl\'{o}, T. and N\'{e}methi, A.: Reduction theorem for lattice cohomology,
  \emph{Preprint}, arXiv:1302.4716v2; to appear in {\it IMRN}. 

\bibitem{Moe}
  Moe, T. K.: Cuspidal curves on Hirzebruch surfaces,
  \emph{PhD Thesis, University of Oslo}, (2013).

\bibitem{Nlattice}
  N\'{e}methi, A.: Lattice cohomology of normal surface singularities,
  \emph{Publ. RIMS. Kyoto Univ.}, \textbf{44} (2008), 507--543.

\bibitem{NT}
  N\'{e}methi, A.:  Two exact sequences for lattice cohomology,
  \emph{Proceedings of the conference organized to honor H. Moscovici's 65th birthday, Contemporary Math.}, \textbf{546} (2011), 249--269.

\bibitem{NJEMS}
  N\'{e}methi, A.: The Seiberg-Witten invariants of negative definite plumbed 3-manifolds,
  \emph{Journal of EMS}, \textbf{13} (4) (2011), 959--974.

\bibitem{NN}  N\'{e}methi, A. and Nicolaescu, L.: Seiberg-Witten invariants and surface singularities,
 {\em Geometry and Topology},  Volume {\bf 6} (2002), 269-328.

\bibitem{NR}
  N\'{e}methi, A. and Rom\'{a}n, F.: The lattice cohomology of $S^3_{-d}(K)$,
  \emph{Zeta Functions in Algebra and Geometry, Contemporary Mathematics},  \textbf{566} (2012), 261--292.


\bibitem{NS}
  N\'{e}methi, A. and Sigur{$\eth$}sson, B.: The geometric genus of hypersurface singularities,
  \emph{Preprint}, arXiv:1310.1268, to appear in {\it JEMS}.

\bibitem{Or}
  Orevkov, S. Y.: On rational cuspidal curves. I. Sharp estimate for degree via multiplicities,
  \emph{Math. Ann.},  \textbf{324} (4) (2002), 657--273.

\bibitem{Pio}
  Piontkowski, J.: On the Number of the Cusps of Rational Cuspidal Plane Curves,
  \emph{Experimental Mathematics},  \textbf{16} (2) (2007), 251--255.

\bibitem{S}
  Stevens, J.: Universal abelian covers of superisolated singularities,
  \emph{Mathematische Nachrichten},  \textbf{282} (8) (2009), 1195--1215.

\bibitem{Tono}
  Tono, K.: On the number of cusps of cuspidal plane curves,
  \emph{Mathematische Nachrichten},  \textbf{278} (1-2) (2005), 216--221.

\end{thebibliography}
\end{document}